\documentclass[12pt]{article}
\usepackage[utf8]{inputenc}
\usepackage{amssymb, amsmath,amsthm,mathtools}
\usepackage{xcolor}
\usepackage{graphicx}
\usepackage[a4paper, margin=1in]{geometry}
\usepackage[shortlabels]{enumitem}
\usepackage{fancyhdr}
\usepackage{hyperref}
\usepackage[capitalise]{cleveref}
\usepackage{tikz}
\usetikzlibrary{cd}

\newtheorem{theorem}{Theorem}
\newtheorem{prop}[theorem]{Proposition}
\newtheorem{lemma}[theorem]{Lemma}
\newtheorem{cor}[theorem]{Corollary}
\newtheorem*{claim}{Claim}
\theoremstyle{remark}
\newtheorem{remark}{Remark}

\newtheorem{example}{Example}
\theoremstyle{definition}
\newtheorem{definition}{Definition}

\setlist[itemize,enumerate]{itemsep=0pt,topsep=0pt}

\renewcommand{\hat}{\widehat}
\renewcommand{\phi}{\varphi}
\renewcommand{\epsilon}{\varepsilon}

\DeclareMathOperator{\supp}{supp}

\newcommand{\Z}{\mathbb Z}

\newcommand{\C}{\mathbb C}

\title{Sampling from the low temperature Potts model through a Markov chain on flows}
\author{Jeroen Huijben\thanks{Korteweg de Vries Institute for Mathematics, Amsterdam. Email \texttt{j.huijben@uva.nl}}, Viresh Patel\thanks{Korteweg de Vries Institute for Mathematics, Amsterdam. Email \texttt{v.s.patel@uva.nl}. Partially supported by Netherlands Organisation for Scientific
Research (NWO) through the Gravitation Programme Networks (024.002.003).}, Guus Regts\thanks{Korteweg de Vries Institute for Mathematics, Amsterdam. Email \texttt{guusregts@gmail.com}. Partially supported by grant VI.Vidi.193.068 from the Dutch Research Council (NWO).}
}

\begin{document}

\maketitle
\begin{abstract}
In this paper we consider the algorithmic problem of sampling from the Potts model and computing its partition function at low temperatures.
Instead of directly working with spin configurations, we consider the equivalent problem of sampling flows.
We show, using path coupling, that a simple and natural Markov chain on the set of flows is rapidly mixing.
As a result we find a $\delta$-approximate sampling algorithm for the Potts model at low enough temperatures, whose running time is bounded by $O(m^2\log(m\delta^{-1}))$ for graphs $G$ with $m$ edges.
\end{abstract}

\footnotesize{{\bf Keywords:} ferromagnetic Potts model, flows, Glauber dynamics, partition function.}

\section{Introduction}
Let $G=(V,E)$ be a graph and let $[q]\coloneqq\{1,\ldots,q\}$ be a set of spins or colours for an integer $q\geq2$. A function $\sigma:V\to [q]$ is called a $q$-spin configuration or colouring.
The Gibbs measure of the $q$-state Potts model on $G=(V,E)$ is a probability distribution on the set of all $q$-spin configurations $\{\sigma:V\to [q]\}$. 
For an \emph{interaction parameter} $w>0$, the Gibbs distribution $\mu_{\mathrm{Potts}}\coloneqq\mu_{\mathrm{Potts},G;q,w}$ is defined by 
\begin{equation}\label{eq:prob pots}
    \mu_{\mathrm{Potts}}[\sigma]\coloneqq\frac{w^{m(\sigma)}}{\sum_{\tau:V \to [q]} w^{m(\tau)}},
\end{equation}
where, for a given $q$-spin configuration $\tau$, $m(\tau)$ denotes the number of edges $\{u,v\}$ of $G$ for which $\tau(u)=\tau(v)$.
The denominator of the fraction~\eqref{eq:prob pots} is called the \emph{partition function of the Potts model} and is denoted by $Z_{\mathrm{Potts}}(G;q,w)$.

The regime $w\in(0,1)$ is known as the anti-ferromagnetic Potts model, and $w \in (1,\infty)$ as the ferromagnetic Potts model. Furthermore, values of $w$ close to 1 are referred to as high temperature, whereas values close to 0 or infinity are referred to as low temperature. This comes from the physical interpretation in which one writes $w= e^{J\beta}$ with $J>0$ being the interaction energy between same spin sites and $\beta$ the inverse temperature.

We will be concerned with the algorithmic problem of approximately sampling from $\mu_{\mathrm{Potts}}$ as well as approximately computing $Z =Z_{\mathrm{Potts}}(G;q,w)$ for $w$ close to infinity (that is in the low temperature ferromagnetic regime).  Given error parameters $\epsilon,\delta \in (0, 1)$, an $\epsilon$-approximate counting
algorithm for $Z$ outputs a number $Z'$ so that $(1- \epsilon) \leq Z/Z' \leq (1 + \epsilon)$, and a $\delta$-approximate sampling
algorithm for $\mu=\mu_\mathrm{Potts}$ outputs a random sample $I$ with distribution $\hat{\mu}$ so that the total variation distance
satisfies  $\|\mu-\hat{\mu}\|_{\mathrm{TV}}\leq \delta$.

It was shown in \cite{GSVY16} that, for graphs of a fixed maximum degree $\Delta \geq 3$, there is a critical parameter $w_\Delta>1$, corresponding to a phase transition of the model on the infinite $\Delta$-regular tree, such that approximating the partition function is computationally hard\footnote{Technically they showed that the problem is \#BIS hard, a complexity class introduced in~\cite{DGGJBIS} and known to be as hard as \#BIS, that is the problem of counting the number of independent sets in a bipartite graph. The exact complexity of \#BIS is unknown, but it is believed that no fully polynomial time randomised approximation scheme exists for \#BIS, but also that \#BIS  is not \#P-hard}. 
This result indicates that it might be hard to compute the partition function of the ferromagnetic Potts model for large values of $w$. 
However, recently several results emerged, showing that for certain finite subgraphs of $\mathbb{Z}^d$~\cite{BR19,HPR20,BCHPT20} as well as $\Delta$-regular graphs satisfying certain expansion properties~\cite{JKP20,HJP20,CGGPSV21} it is in fact possible to approximate the partition function of the ferromagnetic Potts model for $w$ large enough. In fact the algorithms in~\cite{BCHPT20,HJP20} even work for all values $w\geq 1$ under the assumption that the number of colours, $q$, is suitably large in terms of the maximum degree.
The running times of all these aforementioned algorithms are  polynomial in the number of vertices of the underlying graph, but typically with a large exponent. The exception is \cite{CGGPSV21}, in which the cluster expansion techniques of \cite{JKP20} for expander graphs are extended to a Markov chain setting giving running times of the form $O(n^2 \log n)$ for approximating the partition function, where $n$ is the number of vertices of the input graph.

In this paper we present Markov chain based algorithms for approximating the partition function of the ferromagnetic Potts model at sufficiently low temperatures with similar running times as \cite{CGGPSV21}. While most results in this area focus on graphs of bounded maximum degree, the graph parameters of interest for us are different and so our methods, as well as being able to handle subgraphs of the grid $\mathbb{Z}^d$ (although not for all temperatures), can also handle certain graphs classes of unbounded degree (cf. Lemma~\ref{lem:contraction is good}). The parameters of interest for us are in fact similar to those in \cite{BR19}; here we achieve  better running times for our algorithms, while \cite{BR19} achieves better parameter dependencies. 

We show how to efficiently generate a sample from the Potts model using a rapidly mixing Markov chain and then use this to approximate the partition function. The Markov chain however is not supported on $q$-spin configurations\footnote{See e.g.~\cite{BGP16} for an analysis of the usual Glauber dynamics for the ferromagnetic Potts model (at high temperatures).} but on flows taking values in $\Z_q\coloneqq\Z/q\Z$. 
For planar graphs, this Markov chain on flows may be interpreted as Glauber dynamics of $q$-spin configurations on the dual graph; see Section~\ref{se:slow} for an example of this.
We use this Markov chain on flows together with another trick to show that we can efficiently approximate a certain partition function on flows at high temperatures, which in turn can be used to approximate the Potts partition function at low temperatures. Below we state our main results.

\subsection{Main results}
To state our main results, we need some definitions. 
In the present paper we deal with multigraphs and the reader should read multigraph whenever the word graph is used.
A graph is called \emph{even} if all of its vertices have even degree. In what follows we often identify a subgraph of a given graph with its edge set.

Given a graph $G$, fix an arbitrary orientation of its edges. For any even subgraph $C$ of $G$, we can associate to it a signed indicator vector $\chi_C\in \mathbb{Z}^E$ as follows: choose an Eulerian orientation of (each of the components of) $C$. Then for $e\notin C$ we set $\chi_C(e)=0$ and for $e\in C$, we set $\chi_C(e)=1$ if $e$ has the same direction in both $C$ and $G$, and we set $\chi_C(e)=-1$ otherwise. We often abuse notation and identify the indicator vector $\chi_C$ with the set of edges in $C$.
A \emph{$\mathbb{Z}$-flow}, is a map $f:E\to \mathbb{Z}$ satisfying
\[
\sum_{e:\, e \text { directed into } v}f(e)=\sum_{e:\, e \text{ directed out of }v}f(e) \quad \text{ for all $v\in V$}.
\]
We denote the collection of $\mathbb{Z}$-flows by $\mathcal{F}(G)$; note that $\mathcal{F}(G)$ with the obvious notion of addition is known as the first homology group of $G$, and also as the cycle space of $G$. 
Clearly, when viewing $\chi_C$ as a function on $E$, we have $\chi_C\in \mathcal{F}(G)$ for any even subgraph $C$.
It is well known that $\mathcal{F}(G)$ has a generating set (as a $\mathbb{Z}$-module) consisting of indicator vectors of even subgraphs; see e.g. \cite[Section 14]{GRbook}.\footnote{In fact there is even a basis consisting of indicator functions of cycles. For later purposes we however need to work with even subgraphs.}
We call such a generating set an \emph{even generating set} for the cycle space.

Let $\mathcal{C}$ be an even generating set of $\mathcal{F}(G)$; we define some parameters associated to $\mathcal{C}$ (see below for some examples of even generating sets and associated parameters).
For $C\in \mathcal{C}$, let $d(C)\coloneqq|\{D\in \mathcal{C}\setminus \{C\}\mid C\cap D\neq \emptyset\}|$, and let 
\begin{equation}\label{eq:def d}
 d(\mathcal{C})\coloneqq\max \{d(C)\mid C\in \mathcal{C}\}.   
\end{equation}
We write 
\begin{equation}\label{eq:def iota}
\iota(\mathcal{C})\coloneqq\max \{|C_1\cap C_2| \mid C_1,C_2\in \mathcal{C} \text{ with } C_1\neq C_2\}.    
\end{equation} 
Define 
\begin{equation}\label{eq:def ell}
 \ell(\mathcal{C})\coloneqq \max\{|C|\mid C\in \mathcal{C}\}.
\end{equation}

Finally, for an edge $e\in E$, define $s(e)$ to be the number of even subgraphs $C\in \mathcal{C}$ that $e$ is contained in and 
\begin{equation}\label{eq:def s}
s\mathcal(\mathcal{C})\coloneqq\max\{s(e)\mid e\in E\}.    
\end{equation}



We now present our approximate sampling and counting results. All of our results are based on randomised algorithms that arise from running Markov chains. For us, simulating one step of these Markov chains always includes choosing a random element from a set of $t$ elements with some (often uniform) probability distribution, where $t$ is at most polynomial in the size of the input graph. We take the time cost of such a random choice to be $O(1)$ as in the (unit-cost) RAM model of computation; see e.g.\  \cite{MotRag}. 

Our main sampling results read as follows.
\begin{theorem}\label{thm:main sampling}
Fix a number of spins $q\in \mathbb{N}_{\geq 2}$.
\begin{itemize}
\item[(i)] Fix integers $d\geq 2$ and $\iota\geq 1$ and let $\mathcal{G}$ be the set of graphs $G=(V,E)$ for which we have an even generating set $\mathcal{C}$ for $G$ of size $O(|E|)$ such that $d(\mathcal{C}) \leq d$ and $\iota(\mathcal{C})\leq \iota$. 
For any $w>\tfrac{(d+1)\iota }{2}q-(q-1)$ and $\delta \in (0,1)$,
there exists a $\delta$-approximate sampling algorithm for $\mu_{\mathrm{Potts},G;q,w}$, on all $m$-edge graphs  $G \in \mathcal{G}$ 
with running time $O(m^2\log(m\delta^{-1}))$.
\item[(ii)]
Fix integers $\ell\geq 3$ and $s\geq 2$
and let $\mathcal{G}$ be the set of graphs $G=(V,E)$ for which we have an even generating set
$\mathcal{C}$ for $G$ of size $O(|E|)$ such that $\ell(\mathcal{C}) \leq \ell$ and $s(\mathcal{C})\leq s$.
For any $w>(q-1)(\ell s-1)$ and $\delta \in (0,1)$ 
there exists a $\delta$-approximate sampling algorithm for $\mu_{\mathrm{Potts},G;q,w}$ on all $m$-edge graphs $G \in \mathcal{G}$ 
with running time $O(m\log(m\delta^{-1}))$.
\end{itemize}
\end{theorem}
While parts (i) and (ii) are not directly comparable, we note that when $\iota=1$, part (i) has a better range for $w$.

Our main approximate counting results read as follows.
\begin{theorem}\label{thm:main counting}
Fix a number of spins $q\in \mathbb{N}_{\geq 2}$.
\begin{itemize}
\item[(i)]
Fix integers $d\geq 2$ and $\iota\geq 1$ and let $\mathcal{G}$ be the set of graphs $G=(V,E)$ for which we have an even generating set
$\mathcal{C}$ for $G$ of size $O(|E|)$ such that $d(\mathcal{C}) \leq d$ and $\iota(\mathcal{C})\leq \iota$. For $w>\tfrac{(d+1)\iota }{2}q-(q-1)$ and $\epsilon \in (0,1)$,
there exists a randomised $\epsilon$-approximate counting algorithm for $Z_{\mathrm{Potts}}(G;q,w)$ on all $n$-vertex and $m$-edge graphs  $G \in \mathcal{G}$ 
that succeeds with probability at least $3/4$ and has
 running time  $O(n^2m^2\epsilon^{-2}\log(nm\epsilon^{-1}))$.
\item[(ii)]
Fix integers $\ell\geq 3$ and $s\geq 2$ and let $\mathcal{G}$ be the set of graphs $G=(V,E)$ for which we have an even generating set
$\mathcal{C}$ for $G$ of size $O(|E|)$ such that $\ell(\mathcal{C}) \leq \ell$ and $s(\mathcal{C})\leq s$.
For any $w>(q-1)(\ell s-1)$ and $\epsilon \in (0,1)$ 
there exists a randomised $\epsilon$-approximate counting algorithm for $Z_{\mathrm{Potts}}(G;q,w)$ on all $n$-vertex and $m$-edge graphs $G \in \mathcal{G}$ 
that succeeds with probability at least $3/4$ and has running time  $O(n^2m \epsilon^{-2}\log(nm\epsilon^{-1}))$.
\end{itemize}
\end{theorem}
\begin{remark}
We note that the dependence of the Potts model parameter $w$ on the parameters $\ell(\mathcal{C})$ and $s(\mathcal{C})$ is similar as in~\cite{BR19}, except there the dependence on $s$ is order $\sqrt{s}$, which is better than our linear dependence.
This of course raises the question whether our analysis can be improved to get the same dependence.
\end{remark}

We now give a few examples of applications of our results.
\begin{example}
(i) Let $G_1$ and $G_2$ be two graphs that both contain a connected graph $H$ as induced subgraph. Let $G_1\cup_H G_2$ be the graph obtained from $G_1$ and $G_2$ by identifying the vertices of the graph $H$ in both graphs.
If both $G_1$ and $G_2$ have an even generating set consisting of cycles of length at most $\ell$ for some $\ell$, then the same holds for $G_1\cup_H G_2$.

Now use this procedure to build a subgraph $G=(V,E)$ of $\mathbb{Z}^d$, $d\geq 2$ from the union of finitely many copies of elementary cubes ((${\{0,1\}}^d$).
Since an an elementary cube has a generating set consisting of $4$-cycles, as is seen by induction on $d$, and so the resulting graphs has an even generating set $\mathcal{C}$ consisting only of $4$-cycles. The relevant parameters of $\mathcal{C}$ are $d(\mathcal{C})=8(d-1)-4$, $\iota(\mathcal{C})=1$, $\ell(\mathcal{C})=4$, $s(\mathcal{C})=2(d-1)$, and $|\mathcal{C}|\leq (d-1)|E|/2$.

(ii) In a similar manner as in (i) one can also construct graphs with concrete parameters from lattices such as the triangular lattices (here $d=3$, $\iota=1$, $\ell=3$ and $s=2$) or its dual lattice, the honeycomb lattice (here $d=6$, $\iota=1$, $\ell=6$, and $s=2$).

(iii) For any (multi)graph $G=(V,E)$ with even generating set $\mathcal{C}$, the graph $G/e$ obtained by contracting some edge $e \in E$ has an even generating set $\mathcal{C}/e \coloneqq \{C/e: C \in \mathcal{C} \}$. One can check that $\mathcal{C}$ and $\mathcal{C}/e$ have the same parameters $d, \iota, \ell, s$ (see Lemma~\ref{lem:contraction is good}). This allows us to apply our algorithms to many graph classes of unbounded degree e.g.\ any graph that can be obtained from $\mathbb{Z}^d$ by a series of contractions.  
\end{example}

As we shall see in the next subsection, the Markov chains on flows that we introduce are a natural means of studying the ferromagnetic Potts model at low temperatures. The examples above show that it is easy to generate many graphs (also of unbounded degree) for which these chains mix rapidly and therefore for which our results above apply. In this paper, we begin the analysis of these Markov chains on flows, but we believe there is a lot of scope for further study of these chains to obtain better sampling and counting algorithms for the ferromagnetic Potts model at low temperature.

\subsection{Approach and discussion}
The key step in our proof of Theorems~\ref{thm:main sampling} and ~\ref{thm:main counting} is to view the partition function of the Potts model as a generating function of flows taking values in an abelian group of order $q$.
Although well known to those acquainted with the Tutte polynomial and its many specializations, this perspective has not been exploited 
in the  sampling/counting literature (for $q\geq 3$) to the best of our knowledge.
For the special case of the Ising model, that is, $q=2$, this perspective is known as the even `subgraphs world' and has been key in determining an efficient sampling/counting algorithm for the Ising model (with external field) by Jerrum and Sinclair~\cite{JSising}, although the Markov chain used there is defined on the collection of all subsets of the edge set $E$ rather than on just the even sets.
We however define a Markov chain on a state space which, for $q=2$, is supported only on the even sets. 
For $q=2$ one could interpret our Markov chain as Glauber dynamics with respect to a fixed basis of the space of even sets (which forms a vector space over $\mathbb{F}_2$), that is, we move from one even subgraph to another by adding/subtracting elements from the basis. 
In the general case $(q\geq 3)$ the even subgraphs need to be replaced by flows, but, aside from some technical details, our approach remains the same. 
We analyse the Markov chain using the well known method of path coupling~\cite{BD97,Jbook} to obtain our first sampling result Theorem~\ref{thm:main sampling}(i), and
the proof of Theorem~\ref{thm:main counting}(i) then follows by standard arguments after a suitable self-reducibility trick.

Another well known way of representing the partition function of the Potts model is via the random cluster model. 
Only recently, it was shown that a natural Markov chain called random cluster dynamics is rapidly mixing for the Ising model~\cite{GJ18}, yielding another way of obtaining approximation algorithms for the partition function of the Ising model. 
In the analysis a coupling due to Grimmet and Jansson\cite{GM09} between the random cluster model and the even subgraphs world was used.
We extend this coupling to the level of flows and we analyse the Glauber dynamics on the joint space of flows and clusters to obtain a proof of part (ii) of Theorems~\ref{thm:main sampling} and ~\ref{thm:main counting}.

\paragraph{Organization}
In the next section we introduce the notion of flows and the flow partition function, showing the connection to the Potts model and the random cluster model. We also give some preliminaries on Markov chains. In Section~\ref{se:flowchain}, we introduce and analyse the flow chain and prove Theorem~\ref{thm:main sampling}(i). In Section~\ref{se:jointchain} we introduce and analyse the joint flow-random cluster Markov chain, which allows us to prove Theorem~\ref{thm:main sampling}(ii). In Section~\ref{se:count} we examine the subtleties involved in showing that our sampling algorithms imply corresponding counting algorithms: we deduce Theorem~\ref{thm:main counting} from Theorem~\ref{thm:main sampling} in this section. Finally in Section~\ref{se:slow}, we use the duality between flows and Potts configurations to deduce a slow mixing result for our flow chain (on $\mathbb{Z}^2$) from existing results about slow mixing for the Potts model.

\section{Preliminaries}

\subsection{The flow partition function}
Let $G=(V,E)$ be a graph. Throughout, if it is unambiguous, we will take $n\coloneqq|V|$ and $m\coloneqq|E|$. 
In order to define a flow on $G$, we first orient the edges of $G$. (We will assume from now on  that the  edges of graphs have been given a fixed orientation even if this is not explicitly stated). 
For an abelian group $\Gamma$, a $\Gamma$-flow (on $G$) is an assignment $f: E \rightarrow \Gamma$ of a value of $\Gamma$ to every edge of $G$ such that, for every vertex, the sum (in $\Gamma$) at the incoming edges is the same as the sum (in $\Gamma$) at the outgoing edges. 
For a positive integer $q$, the \emph{flow partition function} is defined as
\[
  Z_\mathrm{flow}(G;q,x) = \sum_{f:E \to \mathbb{Z}_q \text{ flow}} x^{\#\text{non-zero edges in }f}.
\]
Note that $Z_\mathrm{flow}$ only depends on the underlying graph and not on the orientation of $G$. 
It is moreover well known that in the definition of the partition function we can replace the group $\mathbb{Z}_q$ by any abelian group $\Gamma$ of order $q$, without changing the partition function. 
We will however make no use of this and solely work with the group $\mathbb{Z}_q$.

Recall from the introduction that $\mathcal{F}(G)$ denotes the set of $\Z$-flows. We write $\mathcal{F}_q(G)$ for the set of $\Z_q$-flows (namely the set of all flows $f:E\to \mathbb{Z}_q$), and for $F\subseteq E$ we denote by $\mathcal{F}_q(V,F)$ the set of all flows $f:F\to \mathbb{Z}_q$. 
The \emph{support} of a flow $f$ is the collection of edges that receive a nonzero flow value and is denoted by $\supp(f)$.
We denote by $\mathrm{nwz}(F;q)$ the number of flows $f:F\to \mathbb{Z}_q$ such that $\supp(f)=F$ (where nwz stands for nowhere zero).
Finally, for positive $x$, there is a natural probability measure $\mu_{\mathrm{flow}}$ on $\mathcal{F}_q(V,E)$, defined  by
\begin{equation}
  \mu_{\mathrm{flow}}(f)\coloneqq\frac{x^{|\supp(f)|}}{Z_{\mathrm{flow}}(G;q,x)}
\end{equation}
for each $f \in \mathcal{F}_q(V,E)$.

The following fact is well known and goes back to Tutte~\cite{Tutte54}.
\begin{lemma}\label{lem:tutte}
Let $q\in \mathbb{N}_{\geq 1}$ and let $x\in \mathbb{C}\setminus\{1\}$. Let $G=(V,E)$ be a graph. Then
\begin{equation}
     q^{|V|} Z_\mathrm{flow}(G;q,x) = (1-x)^{|E|} Z_\mathrm{Potts}\left(G;q,\frac{1+(q-1)x}{1-x}\right).
\end{equation}
\end{lemma}
This lemma follows by combining \eqref{eq:pf Potts=pf RC} and \eqref{eq:lem: pf flow=pf RC} below: it illustrates a useful coupling between random flows and the random cluster model. 
We remark that the function $x\mapsto \frac{1+(q-1)x}{1-x}$ (seen as a function from $\C\cup \{\infty\}\to \C\cup \{\infty\}$) has the property that it sends $0$ to $1$, $1$ to $\infty$ and the interval $[0,1]$ to $[1,\infty]$ in an orientation preserving way.
So approximating the partition function of the $q$-state ferromagnetic Potts model at low temperatures ($1\ll w$) is equivalent to approximating the flow partition function for values $x\in (0,1)$ close to $1$.

\subsection{The random cluster model and a useful coupling}
We view the \emph{partition function of the random cluster model} for a fixed positive integer $q$ as a polynomial in a variable $y$. It is defined for a graph $G=(V,E)$ as follows:
\begin{equation}
    Z_{\mathrm{RC}}(G;q,y)\coloneqq\sum_{F\subseteq E} q^{c(F)}y^{|F|},
\end{equation}
where $c(F)$ denotes the number of components of the graph $(V,F)$.
For $y\geq 0$, we denote the associated probability distribution on the collection of subsets of the edges $\{F\mid F\subseteq E\}$ by $\mu_{\mathrm{RC}}$, i.e., for $F\subseteq E$ we have
\[
\mu_{\mathrm{RC}}(F)=\frac{q^{c(F)}y^{|F|}}{Z_{\mathrm{RC}}(G;q,y)}.
\]
It is well known, see e.g.~\cite{ES88}, that 
\begin{equation}
Z_{\mathrm{Potts}}(G;q,w)=Z_{\mathrm{RC}}(G;q,w-1).
\label{eq:pf Potts=pf RC}
\end{equation}

To describe the connection between $Z_\mathrm{RC}$ and $Z_\mathrm{flow}$ and a coupling between the associated probability distributions, it will be useful to consider the following partition function for a graph $G=(V,E)$:
\begin{equation}
 Z(G;q,x)\coloneqq(1-x)^{|E|}q^{|V|}\sum_{A\subseteq E}\mathrm{nwz}(A;q,x)\sum_{\substack{F\subseteq E \\ A\subseteq F}}\left(\frac{x}{1-x}\right)^{|F|}.  \label{eq:pf general}   
\end{equation}
The associated probability distribution $\mu$ is on pairs $(f,F)$ such that $f$ is a $\mathbb{Z}_q$-flow  on $G$ with $\supp(f)\subseteq F$.
By~\eqref{eq:pf Potts=pf RC}, the next lemma directly implies Lemma~\ref{lem:tutte}; the lemma and the coupling it implies extend~\cite{GM09}.
\begin{lemma}\label{lem: pf flow=pf RC}
Let $q\in \mathbb{N}_{\geq 1}$ and let $x\in \mathbb{C}\setminus\{1\}$. Let $G=(V,E)$ be a graph. 
\begin{equation}
q^{|V|}Z_{\mathrm{flow}}(G;q,x)=Z(G;q,x)=(1-x)^{|E|}Z_{\mathrm{RC}}(G;q,\tfrac{qx}{1-x}). \label{eq:lem: pf flow=pf RC}
\end{equation}
\end{lemma}
\begin{proof}
The first equality follows by the following sequence of identities:
\begin{align*}
&(1-x)^{|E|}q^{|V|}\sum_{A\subseteq E}\mathrm{nwz}(A;q)\sum_{\substack{F\subseteq E \\ A\subseteq F}}\left(\frac{x}{1-x}\right)^{|F|}=q^{|V|}\sum_{A\subseteq E}\mathrm{nwz}(A;q)\sum_{\substack{F\subseteq E \\ A\subseteq F}}x^{|F|}(1-x)^{|E\setminus F|}
\\
&=q^{|V|}\sum_{A\subseteq E}\mathrm{nwz}(A;q)x^{|A|}\sum_{\substack{F\subseteq E \\ A\subseteq F}}x^{|F\setminus A|}(1-x)^{|E\setminus F|}=q^{|V|}\sum_{A\subseteq E}\mathrm{nwz}(A;q)x^{|A|}=q^{|V|}Z_{\mathrm{flow}}(G;q,x).
\end{align*}
For the second equality we use the well known fact that $|\mathcal{F}_q(V,F)|$ (the number of all flows on the graph $(V,F)$ taking values in an abelian group of order $q$), satisfies
\begin{equation}\label{eq:flow count}
|\mathcal{F}_q(V,F)| = Z_\mathrm{flow}((V,F);q,1)=q^{|F|-|V|+c(F)}.
\end{equation}
To see it, note first that we may assume $(V,F)$ is connected since both sides of the identity are multiplicative over components. 
Fix a spanning tree $T\subseteq F$ and assign values from $\mathbb{Z}_q$ to $F\setminus T$. It is not hard to see that these values can be uniquely completed to a flow by iteratively `removing' a leaf from $T$. 

We then have the following chain of equalities:
\begin{align*}
&(1-x)^{|E|}q^{|V|}\sum_{A\subseteq E}\mathrm{nwz}(A;q)\sum_{\substack{F\subseteq E \\ A\subseteq F}}\left(\frac{x}{1-x}\right)^{|F|}=(1-x)^{|E|}\sum_{F\subseteq E}\left(\frac{x}{1-x}\right)^{|F|}q^{|V|}\sum_{A\subseteq F}\mathrm{nwz}(A;q)
\\
&=(1-x)^{|E|}\sum_{F\subseteq E}\left(\frac{x}{1-x}\right)^{|F|}q^{|V|}
|\mathcal{F}_q(V,F)|
=(1-x)^{|E|}\sum_{F\subseteq E}\left(\frac{x}{1-x}\right)^{|F|}q^{|F|+c(F)}
\\
&=(1-x)^{|E|}Z_{\mathrm{RC}}(G;q,\tfrac{qx}{1-x}).
\end{align*}
\end{proof}

The previous lemma in fact gives a coupling between the probability measures $\mu_{\mathrm{flow}}$ and $\mu_{\mathrm{RC}}$ (with the same parameters as in the lemma).
More concretely, given a random flow $f$ drawn from $\mu_{\mathrm{flow}}$ let $A$ be the support of $f$. Next select each edge $e\in E\setminus A$ independently with probability $x$.
The resulting set $F$ is then a sample drawn from $\mu_{\mathrm{RC}}$. To see this, observe that the probability of selecting the set $F$ is given by
\begin{align*}
\sum_{A\subseteq F}\frac{\mathrm{nwz}(A;q)x^{|A|}}{Z_{\mathrm{flow}}(G;q,x)}x^{|F\setminus A|}(1-x)^{|E\setminus F|}
=\frac{|\mathcal{F}_q(V,F)|}{Z_{\mathrm{flow}}(G;q,x)}x^{|F|}(1-x)^{|E\setminus F|}
=\mu_{\mathrm{RC}}(F),
\end{align*}
where the last equality follows by the lemma above and~\eqref{eq:flow count} and the definition of $\mu_{\mathrm{RC}}$.
Conversely (by a similar calculation), given a sample $F$ drawn from $\mu_{\mathrm{RC}}$ one can obtain a random flow drawn from $\mu_{\mathrm{flow}}$ by choosing a uniform flow on $(V,F)$. 

For any $\delta>0$, this procedure transforms a $\delta$-approximate sampler $\hat{\mu_{\mathrm{flow}}}$ for $\mu_{\mathrm{flow}}$ with parameters $q$ and $x\in (0,1)$ into a $\delta$-approximate sampler $\hat{\mu_{\mathrm{RC}}}$ for $\mu_{\mathrm{RC}}$ with parameters $q, \tfrac{qx}{1-x}$ in time bounded by $O(|E|)$.
Indeed, denoting for a flow $f$, $\delta_{\mathrm{flow}}(f)\coloneqq\hat{\mu_{\mathrm{flow}}}(f)-\mu_{\mathrm{flow}}(f)$, we have by the triangle inequality
  \begin{align*}
      \sum_{F\subseteq E} |\hat{\mu_\mathrm{RC}}(F) - \mu_\mathrm{RC}(F)| 
      &= \sum_{F\subseteq E} \left| \sum_{f\in \mathcal{F}_q(V,F)} \delta_\mathrm{flow}(f) x^{|F\setminus \supp(f)|} (1-x)^{|E\setminus F|}\right| \\
      & \leq \sum_{F\subseteq E}  \sum_{f\in  \mathcal{F}_q(V,F)} |\delta_\mathrm{flow}(f)| x^{|F\setminus \supp(f)|} (1-x)^{|E\setminus F|}\\
      &= \sum_{f\in\mathcal{F}_q(V,F) } |\delta_\mathrm{flow}(f)| \sum_{F\supseteq \supp(f)}x^{|F\setminus \supp(f)|} (1-x)^{|E\setminus F|} \\
      &=\sum_{f \in \mathcal{F}_q(V,F) } |\delta_\mathrm{flow}(f)| \leq 2\delta.
  \end{align*}
  
The Edwards-Sokal coupling~\cite{ES88} allows us to generate a sample from the Potts model (with parameters $q, w+1$), given a sample $F$ from the random cluster model (with parameters $q, w$): for each component of $(V,F)$ uniformly and independently choose a colour $i\in [q]$ and colour each of the vertices in this component with this colour. 
Again if we have a $\delta$-approximate sampler $\hat{\mu_{\mathrm{RC}}}$ for $\mu_{\mathrm{RC}}$ this will be transformed into a $\delta$-approximate sampler $\hat{\mu_{\mathrm{Potts}}}$ for $\mu_{\mathrm{Potts}}$ in time bounded by $O(|E|)$.
We summarize the discussion above in a proposition.
\begin{prop} \label{prop:flow sampler to RC/Potts}
Let $G=(V,E)$ be a graph and let  $q\in \mathbb{N}_{\geq 2}$ and $x>0$.
Let $\delta>0$.
Given an approximate $\delta$-approximate sampler $\hat{\mu_{\mathrm{flow}}}$ for $\mu_{\mathrm{flow}}$ with parameters $q$ and $x\in (0,1)$, we can obtain $\delta$-approximate approximate samplers from 
\begin{itemize}
    \item $\mu_{\mathrm{RC}}$ with parameters $q$ and $\tfrac{qx}{1-x}$ in time $O(|E|)$,
    \item $\mu_{\mathrm{Potts}}$ with parameters $q$ and $\tfrac{1+(q-1)x}{1-x}$ in time $O(|E|)$.
\end{itemize}
\end{prop}

\subsection{Generating sets and bases of flows}
In this subsection we give some useful properties of the set of flows and their even generating sets that will allow us to define Markov chains for sampling from $\mu_{\mathrm{flow}}$ in the next section.
In particular we show that an even generating set for the cycle space also generates the collection of $\mathbb{Z}_q$-flows in an appropriate sense to be made precise below.

Let $G=(V,E)$ be a connected graph and recall (from the Introduction) that $\mathcal{F}(G)$ is the set of $\mathbb{Z}$-flows on $G$ and let $\mathcal{C}$ be an even generating set of $\mathcal{F}(G)$. We already mentioned that $\mathcal{F}(G)$ forms a $\mathbb{Z}$-module; in fact it is a free-module of dimension $|E|-|V|+1$, cf.~\cite[Section 14]{GRbook}.
Similarly, the collection of $\Z_q$-flows on $G$ is closed under adding two flows and multiplying a flow by an element of $\Z_q$, making the space of $\Z_q$-flows into a $\Z_q$-module; it is also a free module of dimension $|E|-|V|+1$ by the same argument as for $\mathbb{Z}$, cf.~\cite[Section 14]{GRbook}. (Note that this fact also implies~\eqref{eq:flow count}.)

\begin{lemma}\label{lem:generating set for all}
Let $\mathcal{C}$ be an even generating set for $\mathcal{F}(G)$. Then $\mathcal{C}$ is a generating set for  $\mathcal{F}_q(G)$ for any positive integer $q$.
\end{lemma}
\begin{proof}
Let $f\in\mathcal{F}_q(G)$ be a flow, we will construct a $\mathbb{Z}$-flow $f'$ which reduces modulo $q$ to $f$. Just as in the proof of Lemma~\ref{lem: pf flow=pf RC}, fix a spanning tree $T\subset E$, and now assign to every edge $e\in E\setminus T$ an integer from the residue class $f(e)$. These assignments can be completed iteratively into the flow $f'$ by choosing the edge towards a leaf, assigning a value to satisfy the flow condition in the leaf, and removing the edge from $T$. These new values are also in the residue class prescribed by $f$, because $f$ itself satisfies the flow condition in every leaf encountered. 
Writing $f'$ as a linear combination of $\chi_C$ for $C\in\mathcal{C}$ and reducing modulo $q$, we obtain $f$ as a $\mathbb{Z}_q$-linear combination of $\chi_C$.
\end{proof}


Finally, we will require the following lemma
for our reduction of sampling to counting in Section~\ref{se:count}.
For a graph $G=(V,E)$, a subgraph $H$ of $G$ and an edge $e\in E$, $H/e$ denotes the graph obtained from $H$ by contracting the edge $e$. (If $e$ is not an edge of $H$, then $H/e$ is just $H$.) 

\begin{lemma}\label{lem:contraction is good}
Let $G=(V,E)$ be a graph and let $q\in \mathbb{N}$. Let $\mathcal{C}=\{C_1,\ldots,C_r\}$ be an even generating set for the space of $\mathbb{Z}_q$-flows. 
Let $e\in E$ be a non-loop edge. 
Then $\mathcal{C}'\coloneqq\{C_1/e,\ldots,C_r/e\}$ is an even generating set for the space of $\mathbb{Z}_q$-flows of the graph $G/e$ satisfying $d(\mathcal{C}')\leq d((\mathcal{C})$, $\iota(\mathcal{C}')\leq \iota(\mathcal{C})$, $\ell(\mathcal{C}')\leq \ell(\mathcal{C})$ and $s(\mathcal{C}')\leq s(\mathcal{C})$. 
  \end{lemma}
  \begin{proof}
This follows from the fact that any flow $f'$ on $G/e$ uniquely corresponds to a flow $f$ on $G$. The value on the edge $e$ for $f$ can be read off from the values of the edges incident to the vertex in $G/e$ corresponding to the two endpoints of the edge $e$. 
So, writing $f=\sum_{i=1}^r a_i\chi_{C_i}$ for certain $a_i\in \mathbb{Z}_q$, we get $f'=\sum_{i=1}^r a_i\chi_{C_i/e}$, proving the claim.
The claimed inequalities for the parameters are clear.
\end{proof}
\begin{remark}
Suppose $q$ is a prime in which case $\mathbb{Z}_q$ is a field and $\mathcal{F}_q(G)$ is a vector space over $\Z_q$. 
Then given an even generating set $\mathcal{C}$ for $\mathcal{F}_q(G)$ there exists a basis $\mathcal{C}'$ consisting only of cycles for which the parameters $d,\iota,\ell$ and $s$ are all not worse.  
To see this note that if $\mathcal{C}$ is a generating set and not a basis, we can always remove elements from it to make it into a basis.
If $\mathcal{C}$ forms a basis and some $C\in \mathcal{C}$ is the edge disjoint union of two nonempty even subgraphs $K_1$ and $K_2$, we have that either $(\mathcal{C}\setminus \{C\})\cup \{K_1\}$ or $(\mathcal{C}\setminus \{C\})\cup \{K_2\}$ forms a basis. 
This is generally not true for composite $q$ and therefore we work with even generating sets.
\end{remark}




\subsection{Preliminaries on Markov chains}

To analyse the mixing time of our Markov chains, we will use the path coupling technique. We briefly recall the following results from Section 2 in \cite{DG98}.

Let $\mathcal{M} = (Z_t)_{t=0}^\infty$ be an ergodic, discrete-time Markov chain on a finite state space $\Omega$ with transition matrix $P$. Let $\mu_t$ be the distribution of $Z_t$ and let $\mu$ be the (unique) stationary distribution of $\mathcal{M}$.
Two distributions on $\Omega$ are said to be \emph{$\delta$-close} if the total variation distance between them is at most $\delta$. The \emph{$\delta$-mixing time} of $\mathcal{M}$ is the minimum number of steps after which $\mathcal{M}$ is $\delta$-close to its stationary distribution (i.e.\ the smallest $t$ such that $\|\mu_t - \mu \|_{\mathrm{TV}} \leq \delta$). 

A \emph{coupling} for $\mathcal{M}$ is a stochastic process $(X_t,Y_t)$ on $\Omega^2$, such that  each of $X_t$ and $Y_t$, considered independently, transition according to $P$. More precisely, the coupling can be defined by its transition matrix $P'$: given $(x, y)$ and $(x',y') \in \Omega^2$, $P'((x,y),(x',y'))$ is the probability that $(X_{t+1}, Y_{t+1}) = (x',y')$ given that $(X_t,Y_t) = (x,y)$. For $P'$ to describe a valid coupling, it must satisfy for each $(x,y) \in \Omega^2$, that
\begin{align}
\sum_{y' \in \Omega}P'((x,y),(x',y')) &= P(x,x') \quad \text{ for all } x' \in \Omega; \nonumber \\
\sum_{x' \in \Omega}P'((x,y),(x',y')) &= P(y,y') \quad \text{ for all } y' \in \Omega  .
\label{eq:coupling}
\end{align}



For our use of path coupling, we require an integer-valued distance function $d$ on $\Omega$ such that between any two states $x,y\in\Omega$ there exists a sequence $x=x_0,x_1,\dots,x_s=y$ in which consecutive states are at distance $1$. 
If we can define a coupling on the set of pairs $(x,y) \in \Omega^2$ for which $d(x,y) =1$.  (that is, we define transition probabilities $P'((x,y),(x',y'))$ for all $(x,y)$ such that $d(x,y)=1$, and $(x',y')\in \Omega^2$ that satisfy equations \eqref{eq:coupling}) then this can be extended to a complete coupling on $\Omega^2$.  We can use such a (partial) coupling to bound the mixing time of $\mathcal{M}$ via the following result:


\begin{theorem}[Theorem 2.2 in \cite{DG98}]\label{thm:path coupling}
Let $\mathcal{M}$ be a Markov chain on $\Omega$ and $d$ an integer-valued distance on $\Omega$ as above with maximum distance $D$. Assume there is a coupling $(X_t,Y_t) \mapsto (X_{t+1},Y_{t+1})$ defined for all pairs with $d(X_t,Y_t)=1$ (as described above) such that
\[
\mathbb{E}(d(X_{t+1},Y_{t+1}) \mid (X_t,Y_t)) \leq 1-\alpha
\]
for some $\alpha>0$. Then the Markov chain $\mathcal{M}$ has $\delta$-mixing time at most $\frac{\log(D\delta^{-1})}{\alpha}$.
\end{theorem}

\section{Flow Markov chain}
\label{se:flowchain}

In this section, we introduce and analyse the flow Markov chain and use it to prove Theorem~\ref{thm:main sampling}(i).

\begin{definition}
Let $G=(V,E)$ be a graph and $\mathcal{C}$ an even generating set of $\mathcal{F}_q(V,E)$ of size $r$.
The \emph{flow Markov chain} for $(G,\mathcal{C})$ is a Markov chain on the state space $\mathcal{F}_q(V,E)$. For every flow $f\in\mathcal{F}_q(V,E)$, $t\in\mathbb{Z}_q\setminus\{0\}$ and $C\in\mathcal{C}$, the transition probabilities of the Markov chain are given by:
\begin{align*}
P_\mathrm{flow}(f,f+t\chi_C) &=\frac{1}{r} \frac{\mu_\mathrm{flow}(f+t\chi_C)}{\sum_{u\in\Z_q} \mu_\mathrm{flow}(f+u\chi_C)},\\
P_\mathrm{flow}(f,f) &= \frac{1}{r}\sum_{C\in\mathcal{C}}\frac{\mu_\mathrm{flow}(f)}{\sum_{u\in\Z_q} \mu_\mathrm{flow}(f+u\chi_C)},
\end{align*}
and all other transition probabilities are zero.
\end{definition}
We see easily that the measure $\mu_\mathrm{flow}$ satisfies the detailed balance equation \[\mu_\mathrm{flow}(f) P_\mathrm{flow}(f,f+t\chi_C) = \mu_\mathrm{flow}(f+t\chi_C) P_\mathrm{flow}(f+t\chi_C,f),\] so $\mu_\mathrm{flow}$ is the stationary distribution of the flow Markov chain.

We can simulate one step of this Markov efficiently by first selecting $C\in\mathcal{C}$ uniformly at random, and for $t\in\Z_q$, selecting $f+t\chi_C$ with probability proportional to
\[ \mu_\mathrm{flow}(f+t\chi_C)/\mu_\mathrm{flow}(f) = x^{\#\{e\in C \mid f(e)=0\} - \#\{e\in C \mid f(e)+t\chi_C(e)=0\}}.
\]
For fixed $q$, simulating one step of the Markov chain requires $O(\ell)$ time (where $\ell = \max_{C \in \mathcal{C}}|C|$) in order to compute $f + t \chi_C$ and its support. We bound this by $O(m)$.

\subsection{Rapid mixing of flow Markov chain}
\begin{theorem}\label{thm:mixing time flow MC}
Let $q,d\geq2, \iota\geq1$ be integers and $1>x>1-\frac{2}{(d+1)\iota}$.
Write $\xi=x-\left(1-\frac{2}{(d+1)\iota}\right)$ and let $\delta>0$. 
Now let $G=(V,E)$ be a graph and $\mathcal{C}$ an even generating set of $\mathcal{F}_q(G)$ of size $r$ satisfying $d(\mathcal{C})\leq d$ and $\iota(\mathcal{C})\leq \iota$,
then the $\delta$-mixing time of the flow Markov chain for $(G,\mathcal{C})$ with parameter $x$ is at most $\frac{4r}{d\iota}\log(r \delta^{-1})\xi^{-1}$.
\end{theorem}

\begin{remark}
\label{re:runtime}
Because $\xi <\frac{2}{(d+1)\iota}\leq\frac{2}{d\iota}$, the upper bound in this Theorem is always at least $2r\log(r\delta^{-1})$. This shows the upper bound doesn't get better with larger $d$ and $\iota$, even though they are in the denominator.
\end{remark}

For the given range of $x$, the flow Markov chain therefore gives  an efficient, randomised algorithm for approximately sampling flows according to $\mu_\mathrm{flow}$.
Combining this with \cref{prop:flow sampler to RC/Potts}, we obtain the following Corollary; it directly implies \cref{thm:main sampling}(i) by \cref{lem:generating set for all}.

 \begin{cor} 
  Fix integers $q,d\geq 2$ and $\iota\geq 1$. For any $w>\frac{(d+1)\iota}{2}q-(q-1)$ and $\delta>0$, there exists an algorithm that on input of an $m$-edge graph $G$ and even generating set $\mathcal{C}$ of $\mathcal{F}_q(G)$ of size $r$ satisfying $d(\mathcal{C})\leq d$ and $\iota(\mathcal{C})\leq \iota$ outputs a $q$-state Potts colouring $\sigma:V\to [ q ] $ within total variation distance $\delta$ of the $q$-state Potts-measure $\mu_\mathrm{Potts}$ with parameter $w$. 
  This is obtained by running the flow Markov chain for  at most $O(r\log(r \delta^{-1}))$ steps where each step takes $O(m)$ time.
\end{cor}

The following technical lemma will be used in the proof of \cref{thm:mixing time flow MC}. Note that the lower bound is actually attained in the limit case $(a_1,\ldots,a_q)=(\iota,0,-\infty,\ldots,-\infty), (b_1,\ldots,b_q)=(0,\iota,-\infty,\ldots,-\infty)$. The proof is postponed to the end of this section.
\begin{lemma}\label{lem:sum of minima}
  Let $x\in(0,1)$ be a real number, and $\iota\geq0$ and $a_1,\ldots,a_q,b_1,\ldots,b_q$ integers satisfying the following constraints:
\begin{itemize}
  \item $\sum_i a_i = \sum_i b_i$;
  \item $\sum_i |a_i-b_i| \leq 2\iota$.
\end{itemize}
Then
\[
  S \coloneqq \sum_i \min\left(\frac{x^{-a_i}}{\sum_j x^{-a_j}},\frac{x^{-b_i}}{\sum_j x^{-b_j}}\right) \geq 1- \frac{1-x^\iota}{1+x^\iota}.
\]
\end{lemma}

\begin{proof}[Proof of \cref{thm:mixing time flow MC}]
  To prove the theorem we determine an upper bound for the mixing time of the flow Markov chain by using path coupling. For this we define the distance between two flows as the minimal number of steps the flow Markov chain needs to go from one to the other. By \cref{thm:path coupling} it is now enough to define a coupling for states at distance 1. If the expected distance after one step of this coupling is at most $1-\alpha$, the mixing time of the Markov chain is at most $T \coloneqq \frac{\log(r\delta^{-1})}{\alpha}$. (The maximal distance in $\mathcal{F}_q(V,E)$ is at most $r$, because in $r$ steps the coefficients of every even set in $\mathcal{C}$ can be adjusted to the desired value.)
  
  We will construct a coupling on states at distance $1$ for which $\alpha=\frac{(d+1)x^{\iota}-(d-1)}{2r} \geq \frac{d\iota}{4r}\xi$. Therefore the running time of the sampler is bounded by $T\leq \frac{4r}{d\iota}\log(r\delta^{-1})\xi^{-1}$ steps of the flow Markov chain.
  
  Consider a pair of flows $(f,g)$ which differ by a multiple of $\chi_C$. To construct the coupling we first select u.a.r.\@ an even set $D\in\mathcal{C}$. We will separate three cases, and define the transition probabilities in each of these cases. The cases are (a) when $C=D$, (b) when $C$ and $D$ have no common edges, and (c) when $C$ and $D$ do have common edges.
  
  \begin{enumerate}[(a)]
      \item We get a valid coupling by making the transition $(f,g) \to (f+t\chi_D,f+t\chi_D)$ with probability $\frac{\mu_\mathrm{flow}(f+t\chi_D)}{\sum_{u\in\Z_q} \mu_\mathrm{flow}(f+u\chi_D)}$. Then the distance will always drop from 1 to 0.
      \item Now the edges of $D$ have the same values in $f$ and $g$, and we see that $\mu_\mathrm{flow}(f+t\chi_D)/\mu_\mathrm{flow}(f) = \mu_\mathrm{flow}(g+t\chi_D)/\mu_\mathrm{flow}(g)$ for all $t$. Therefore we get a valid coupling by making the transition $(f,g) \to (f+t\chi_D , g+t\chi_D)$ with probability $\frac{\mu_\mathrm{flow}(f+t\chi_D)}{\sum_{u\in\Z_q} \mu_\mathrm{flow}(f+u\chi_D)}=\frac{\mu_\mathrm{flow}(g+t\chi_D)}{\sum_{u\in\Z_q} \mu_\mathrm{flow}(g+u\chi_D)}$. In this case the distance between the two new states remains 1.
      \item The coupling in this case is more complicated, as the values of $f$ and $g$ on $D$ are different. Below we prove the following:
      \begin{claim}
      There is a coupling where the total probability for all transitions $(f,g) \to (f+t\chi_D,g+t\chi_D)$ is at least $1-\frac{1-x^\iota}{1+x^\iota}$.
      \end{claim}
      In all these transitions the distance remains 1, and therefore the probability of the distance increasing to 2 is at most $\frac{1-x^\iota}{1+x^\iota}$.
  \end{enumerate}
  
  We can now calculate the expected distance after one step of this coupling. Case (a) occurs with probability $1/r$, and case (c) with probability at most $d/r$. Hence the expected distance is at most
  \begin{align*}
      &1-\frac{1}{r} + \frac{d}{r}\cdot\frac{1-x^\iota}{1+x^\iota}\\
      &=1- \frac{1+x^\iota-d(1-x^\iota)}{r(1+x^\iota)}\\
      &=1- \frac{(d+1)x^\iota-(d-1)}{r(1+x^\iota)}\\
      &\leq 1- \frac{(d+1)x^\iota-(d-1)}{2r} = 1-\alpha.
  \end{align*}
  We see that $\alpha$ is positive for $x>1-\frac{2}{(d+1)\iota}>\sqrt[\iota]{1-\frac{2}{d+1}}= \sqrt[\iota]{\frac{d-1}{d+1}}$. Further, we see for these $x$ that the derivative of $\alpha$ with respect to $x$ satisfies,
  \[
  \frac{\mathrm{d}\alpha}{\mathrm{d}x} = \frac{(d+1)\iota x^{\iota-1}}{2r} \geq \frac{(d+1)\iota x^\iota}{2r} \geq \frac{(d-1)\iota}{2r} \geq \frac{d\iota}{4r}.
  \]
  Hence we find that $\alpha \geq \frac{d\iota}{4r}\xi$.

We finish by proving the Claim in case (c).
Explicitly the transition probabilities in this case are given by (writing $p_t=\frac{\mu_\mathrm{flow}(f+t\chi_D)}{\sum_{u\in\Z_q} \mu_\mathrm{flow}(f+u\chi_D)}$ and $q_t=\frac{\mu_\mathrm{flow}(g+t\chi_D)}{\sum_{u\in\Z_q} \mu_\mathrm{flow}(g+u\chi_D)}$)
\[
(f,g) \to (f+t\chi_D,g+t\chi_D) \text{ with probability } \min(p_t,q_t),
\]
and for $s\neq t$
\begin{align*}
(f,g)\to(f+s\chi_D,g+t\chi_D) \text{ with probability  } &\frac{(p_s-\min(p_s,q_s))(q_t-\min(p_t,q_t))}{\sum_{u\in\Z_q} (p_u-\min(p_u,q_u))} 
\\
&=\frac{(p_s-\min(p_s,q_s))(q_t-\min(p_t,q_t))}{\sum_{u\in\Z_q} (q_u-\min(p_u,q_u))}.
\end{align*}
It is easily checked that this yields a valid coupling, i.e.\@ that the first coordinate has transition probabilities $p_t$, and similary $q_t$ for the second coordinate.

Now we wish to bound the sum of the diagonal entries.
To do this we have to take a closer look at the weights occurring in this table. We define $a_i$ to be the number of edges in $D$ with value 0 in the flow $f+i D$. This ensures that $\mu_\mathrm{flow}(f+i \chi_D) \propto x^{-a_i}$ and $p_t=\frac{x^{-a_t}}{\sum_u x^{-a_u}}$. Similarly, we define $b_i$ as the number of edges in $D$ with value 0 in the flow $g+i \chi_D$. 

We derive some boundary conditions on the $a_i$'s and $b_i$'s. Ranging $i$ over $\Z_q$, every edge of $D$ will get value $0$ in exactly one of $f+i \chi_D$. So $\sum_i a_i$ is the length $|D|$. The same holds for the $b_i$'s, so in particular we find that $\sum_i a_i = \sum_i b_i$.

Second we will bound $\sum_i |a_i-b_i|$. If an edge is counted in $a_i$, but not in $b_i$, it must be an edge of $C$. For every such edge it can happen once that it is counted in $a_i$ and not $b_i$, and once vice versa. Hence the total absolute difference $\sum_i |a_i-b_i|$ is bounded by $2|C \cap D|\leq 2\iota$.

Now the sum of all the probabilities on the diagonal is
\[
\sum_i \min \left( \frac{x^{-a_i}}{\sum_j x^{-a_j}} , \frac{x^{-b_i}}{\sum_j x^{-b_j}} \right),
\]
and the numbers $a_i,b_i$ satisfy the conditions of \cref{lem:sum of minima}, so the sum is bounded below by $1-\frac{1-x^\iota}{1+x^\iota}$.\qedhere

\end{proof}




\begin{proof}[Proof of \cref{lem:sum of minima}]
  First of all, let us introduce a little terminology: an index $i$ is called $b$-minimal if the minimum of the $i$-term in $S$ is not equal to the $a$-term. Also assume that $\sum_j x^{-a_j} \geq \sum_j x^{-b_j}$. And note that the two conditions imply
  \[
    2\iota \geq \sum_i |a_i-b_i| \geq |a_j-b_j| + \left| \sum_{i\neq j} a_i-b_i \right| = |a_j-b_j| + |b_j-a_j|=2|a_j-b_j|.
  \]
Hence the absolute difference between $a_j$ and $b_j$ is always at most $\iota$.

  The proof contains two steps. In the first step, we change the numbers $a_i$ in such a way that the conditions still hold and $S$ does not increase. After the first step there will be at most one $b$-minimal index $i$. This allows us to eliminate the minima from the expression for $S$. In the second step, we give a lower bound for this new obtained expression.

  For the first step, assume that two different indices $t,u$ are $b$-minimal, and assume also that $a_t\geq a_u$. Now we increase $a_t$ by 1, and decrease $a_u$ by 1, i.e.\@ define the new sequence
  \[
    a'_i=\begin{cases}a_t+1 & i=t, \\ a_u-1 & i=u, \\ a_i & \text{otherwise}.\end{cases}
  \]
    First we note that $\sum_j x^{-a_j'} > \sum_j x^{-a_j}$, simply because 
  \[
    x^{-a'_t}-x^{-a_t}=x^{-(a_t+1)}(1-x) > x^{-a_u}(1-x) = x^{-a_u} - x^{-a'_u}.
  \]

  Now we will show for every $i$, that the term $\min(x^{-a_i}/\sum_j x^{-a_j},x^{-b_i}/\sum_j x^{-b_j})$ does not increase. For $i\neq t,u$ this is easy, because $x^{-a_i}$ does not change and the sum in the denominator increases. Hence the first term in the minimum decreases and the minimum cannot increase. We also assumed that both $t,u$ were $b$-minimal, and because we don't change the $b_i$'s, the minimum cannot increase.

  Further, we have to check that the new sequence still satisfies all the conditions. It is clear that $\sum_i a'_i = \sum_i a_i = \sum_i b_i$ and $\sum_j x^{-a_j'} > \sum_j x^{-a_j} \geq \sum_j x^{-b_j}$. Further we see for $i=t,u$ that
  \[
    \frac{x^{-a_i}}{ \sum_{j} x^{-a_j}} > \frac{x^{-b_i}}{ \sum_{j} x^{-b_j}} \geq \frac{x^{-b_i}}{ \sum_{j} x^{-a_j}},
  \]
  hence $a_i > b_i$ for $i=t,u$. Therefore $|a'_t-b_t|= |a_t-b_t +1|= |a_t-b_t|+1$ and $|a'_u-b_u|=|a_u-b_u-1|=|a_u-b_u|-1$ (because $a_u-b_u$ is a positive integer), so the sum of the absolute values remains the same.

  After repeating this adjustment with the same indices, eventually one of them will stop being $b$-minimal. Now repeat with two new $b$-minimal indices, as long as they exist. In the end there must be at most one $b$-minimal index.

  Now we are ready for step two. If there are no $b$-minimal indices, the sum is equal to 1 and the result holds. Hence we assume wlog that 1 is the only $b$-minimal index and we can write
  \[
    S= \frac{x^{-b_1}}{\sum_j x^{-b_j}} + \sum_{i\neq 1} \frac{x^{-a_i}}{\sum_j x^{-a_j}} =\frac{x^{-b_1}}{\sum_j x^{-b_j}} + 1-\frac{x^{-a_1}}{\sum_j x^{-a_j}}.
  \]
  Note that for positive $p,q$, the function $ \frac{-p}{p+q} $ is increasing in $q$ and decreasing in $p$. Because $x^{-a_1} \leq x^{-(b_1+\iota)}$ and $\sum_{j\geq2} x^{-a_j} \geq \sum_{j\geq2} x^{-(b_j-\iota)}$, we can thus estimate that
  \[
    S \geq \frac{x^{-b_1}}{\sum_j x^{-b_j}} + 1 - \frac{x^{-\iota} x^{-b_1}}{x^{-\iota} x^{-b_1}+x^\iota \sum_{j\geq2} x^{-b_j}}.
  \]
  Now write $X=x^{-\iota}$, $B_1=x^{-b_1}$ and $B_2=\sum_{j\geq2} x^{-b_j}$, so that the lower bound for $S$ can be written as $ \frac{B_1}{B_1+B_2} +1- \frac{X^2B_1}{X^2B_1+B_2}$. By AM-GM we can estimate that
\[
  (B_1+B_2)(X^2B_1+B_2)=X^2B_1^2+B_2^2+(X^2+1)B_1B_2 \geq 2XB_1B_2+(X^2+1)B_1B_2 = (X+1)^2B_1B_2,
\]
so that we find:
  \begin{align*}
    S \geq 1+ \frac{B_1}{B_1+B_2} - \frac{X^2B_1}{X^2B_1+B_2}  &= 1+ \frac{B_1(X^2B_1+B_2)-X^2B_1(B_1+B_2)}{(B_1+B_2)(X^2B_1+B_2)} \\
						     &= 1- \frac{(X-1)(X+1)B_1B_2}{(B_1+B_2)(X^2B_1+B_2)} \\
						     &\geq 1- \frac{(X-1)(X+1)B_1B_2}{(X+1)^2B_1B_2} = 1- \frac{X-1}{X+1}. \qedhere
  \end{align*}
\end{proof}

\section{Joint flow-random cluster Markov chain}
\label{se:jointchain}

In this section we will consider a different chain that allows us to sample flows. We will again prove rapid mixing by using path coupling, and this holds for roughly the same range of parameters $x$.

To describe the chain let $q\geq 2$ be an integer and let $G=(V,E)$ be a graph $m$ edges. Let $\mathcal{C}$ be an even generating set  for the flow space $\mathcal{F}_q(G)$ of size $r$ and let $\ell=\ell(\mathcal{C}).$

\begin{definition}
Let $\Omega_{\mathrm{flow-RC}}$ be the set of pairs $(f,F)$ with $F\subset E$ a set of edges and $f$ a flow on $(V,F)$.
The \emph{joint flow-RC Markov chain} is a Markov chain on the state space $\Omega_{\mathrm{flow-RC}}$ depending on two parameters $x,p\in(0,1)$. The transition probabilities are as follows:\\
For $e\in E \setminus F$:
\[
P_{\mathrm{flow-RC}}[(f,F) , (f,F\cup\{e\})] = \frac{(1-p)x}{m}.
\]
For $e\in F$ such that $f(e)=0$:
\[
P_{\mathrm{flow-RC}}[(f,F) , (f,F\setminus\{e\})] = \frac{(1-p)(1-x)}{m}.
\]
And for $t\in\{1,\ldots,q-1\}$, $C\in\mathcal{C}$ an even set such that $C\subseteq F$:
\[
P_{\mathrm{flow-RC}}[(f,F) , (f+t\chi_C,F)] = \frac{p}{qr}.
\]
All other transition probabilities are zero, except for the stationary probabilities $P_{\mathrm{flow-RC}}[(f,F),(f,F)]$.
\end{definition}


Simulating one step of this Markov chain starting in the state $(f,F)$ can be done as follows. We first select either `flow' or `edges' with probabilities resp.\@ $p$ and $1-p$.
\begin{itemize}
    \item If we select `flow', we will update the flow $f$. We choose $C \in \mathcal{C}$ and $t\in\mathbb Z_q$ uniformly at random. If the flow $f + t\chi_C$ is supported on $F$ (for $t\neq 0$ this is equivalent to $C\subseteq F$), we make the transition $(f,F) \to (f+t\chi_C,F)$. Otherwise the chain stays in $(f,F)$.
    \item If we select `edges', we will update the set of edges $F$. We choose an edge $e\in E$ uniformly at random. If $e$ is not contained in $F$, we make with probability $x$ the transition $(f,F)\to(f,F\cup\{e\})$. If $e$ is contained in $F$ and $f(e)=0$, we make with probability $1-x$ the transition $(f,F) \to (f, F\setminus\{e\})$. Otherwise the chain stays in $(f,F)$.
\end{itemize}
The total cost of simulating one step of this Markov chain is $O(\ell)$ for checking whether $C \subseteq F$ in the first case.

Further this Markov chain has stationary distribution $\mu_{\mathrm{flow-RC}} \colon (f,F) \mapsto \frac{1}{Z_\mathrm{flow}} x^{|F|}(1-x)^{|E\setminus F|}$. (From \cref{lem: pf flow=pf RC} it follows that the sum over all states is 1.) This follows easily from checking the detailed balance equation.


\subsection{Rapid mixing of joint flow-RC Markov chain}

\begin{theorem}\label{thm:radid mixing joint}
Let $\ell\geq3, q,s\geq2$ be integers and $1>x>1-\frac{q}{(q-1)\ell s}$.
Write $\xi=x-\left(1-\frac{q}{(q-1)\ell s}\right)$ and let $\delta>0$.
Let $G=(V,E)$ be a graph and $\mathcal{C}$ an even generating set of $\mathcal{F}_q(V,E)$ of size $r$ satisfying $\ell(\mathcal{C}) \leq \ell$ and $s(\mathcal{C}) \leq s$, then there is a value of $p$ for which the joint flow-RC Markov chain for $(G,\mathcal{C})$ comes $\delta$-close to $\mu_{\mathrm{flow-RC}}$ with parameter $x$ in at most $\frac{2(m+r)}{\ell}\log((2m+r)\delta^{-1})\xi^{-1}$ steps.
\end{theorem}
\begin{remark}
An exact value for $p$ in the theorem above can be obtained from equation \eqref{eq:determine p} below.
\end{remark}
\begin{remark}
Note again that $\xi > \frac{q}{(q-1)\ell s} > \frac{1}{\ell s}$, and hence the required number of calls in the above theorem is at least $2s(m+r) \log((2m+r)\delta^{-1})$. Again this means the bound does not get better with larger $\ell$, even though it appears in the denominator, and even gets worse with larger $s$.
\end{remark}

It would be interesting to see if the theorem could be used to say anything about possible rapid mixing of the Glauber dynamics for the random cluster model at low temperatures cf.~\cite{GJ18}.

The following corollary is immediate by Proposition~\ref{prop:flow sampler to RC/Potts} and directly implies Theorem~\ref{thm:main sampling}(ii) by \cref{lem:generating set for all}.
\begin{cor}
Fix integers $\ell\geq 3$ and $q,s\geq 2$.
Let $w>(q-1)(\ell s-1)$ and $\delta>0$, then there exists an algorithm that on input an $m$-edge graph $G$ and an even generating set $\mathcal{C}$ for $\mathcal{F}_q(G)$ of size $r$ satisfying $\ell(\mathcal{C}) \leq \ell$ and $s(\mathcal{C})\leq s$, outputs a $q$-state Potts colouring $\sigma:V\to [ q ] $ within total variation distance $\delta$ of the $q$-state Potts-measure $\mu_\mathrm{Potts}$ with parameter $w$. This is obtained  by running the joint flow-RC Markov chain for $O((m+r)\log((m+r)\delta^{-1}))$ steps, where each step takes $O(1)$ time (since $\ell$ is fixed).
\end{cor}

\begin{proof}[Proof of \cref{thm:radid mixing joint}]

We will again use path coupling to deduce rapid mixing of the above defined Markov chain. The distance we use on the state space is defined as the least number of steps required in the Markov chain to go from one state to the another. 
A crude upper bound on the diameter is given by $2m+r$.
There are two kinds of pairs of states at distance one, which we will treat separately. Just as in the proof of \cref{thm:mixing time flow MC}, we will prove that the expected distance after one step of the coupling is at most $1-\alpha$ for some $\alpha$, and therefore the mixing time is at most $\log((2m+r)\delta^{-1})\alpha^{-1}$.

Consider the states $(f,F)$ and $(f,F \cup \{e\})$. We will make a coupling on them. The transition probabilities of this coupling are as follows:
%
\[
\begin{pmatrix*}[l] f & F\\ f  & F\cup\{e\} \end{pmatrix*} \to
\begin{cases}
\begin{pmatrix*}[l] f & F\cup\{e\} \\ f & F\cup\{e\}\end{pmatrix*} & \frac{(1-p)x}{m},\\
\begin{pmatrix*} f & F \\ f & F \end{pmatrix*} & \frac{(1-p)(1-x)}{m},\\
\begin{pmatrix*}[l] f & F\cup\{e'\} \\ f &  F\cup\{e,e'\} \end{pmatrix*} & \frac{(1-p)x}{m} \text{ if $e'\not\in F\cup\{e\}$}, \\
\begin{pmatrix*}[l] f & F\setminus\{e'\} \\ f & F\setminus\{e'\} \cup \{e\} \end{pmatrix*} & \frac{(1-p)(1-x)}{m} \text{ if $e'\in F$ and $f(e')=0$},\\

\begin{pmatrix*}[l] f+t\chi_C & F \\ f+t\chi_C & F\cup\{e\} \end{pmatrix*} & \frac{p}{qr} \text{ if $t\neq0$ and $C\subseteq F$},\\
\begin{pmatrix*}[l] f & F \\ f+t\chi_C & F\cup\{e\} \end{pmatrix*} & \frac{p}{qr} \text{ if $t\neq0$, $e\in C$ and $C\subseteq F\cup\{e\}$}.
\end{cases}
\]
The first two cases each occur exactly once and decrease the distance by one. The last case occurs at most $s(q-1)$ times and increases the distance by one. Therefore the expected distance after one step of the coupling is at most
\[
1-\frac{1-p}{m}+\frac{(q-1)sp}{qr}
\]
in this case.

Next is the coupling on the neighbouring states $(f,F)$ and $(f+t\chi_C,F)$ (with $t \not= 0$). The transition probabilities are as follows:
\[
\begin{pmatrix*}[l]f & F \\ f+t\chi_C & F \end{pmatrix*} \to
\begin{cases}
\begin{pmatrix*}[l]f & F\cup\{e\} \\ f+t\chi_C & F\cup\{e\} \end{pmatrix*} & \frac{(1-p)x}{m} \text{ if $e\not\in F$},\\
\begin{pmatrix*}[l]f & F\setminus\{e\} \\ f+t\chi_C & F\setminus\{e\} \end{pmatrix*} & \frac{(1-p)(1-x)}{m} \text{ if $e\not\in F$ and $e\not\in C$},\\
\begin{pmatrix*}[l]f & F \setminus \{e\}\\ f+t\chi_C & F \end{pmatrix*} & \frac{(1-p)(1-x)}{m} \text{ if $e\in C$ and $f(e)=0$},\\
\begin{pmatrix*}[l]f & F\\ f+t\chi_C & F\setminus\{e\} \end{pmatrix*} & \frac{(1-p)(1-x)}{m} \text{ if $e\in C$ and $f(e)+t\chi_C(e)=0$},\\
\begin{pmatrix*}[l]f+t'\chi_C & F \\f+t'\chi_C & F \end{pmatrix*} & \frac{p}{qr},\\
\begin{pmatrix*}[l]f+t'\chi_{C'} & F \\ f+t\chi_C+t'\chi_{C'} & F\end{pmatrix*} & \frac{p}{qr} \text{ if $t'\neq 0$ and $C'\neq C$}.
\end{cases}
\]
The third and fourth case occur together at most $\ell$ times and increase the distance with one. The fifth case occurs exactly $q$ times and decreases the distance with one. Therefore the expected distance after one step of the coupling is at most
\[
1-\frac{p}{r}+\frac{\ell(1-x)(1-p)}{m}.
\]

To find a useful coupling, both expected distances will have to be smaller than one and we have to solve the following equations (for $p$ and $\alpha$):
\begin{align}
1-\frac{1-p}{m}+\frac{(q-1)sp}{qr} &= 1-\frac{p}{r}+\frac{\ell(1-x)(1-p)}{m} = 1-\alpha, \nonumber
\\
\text{i.e. }\:\frac{1-p}{m}-\frac{(q-1)sp}{qr} &= \frac{p}{r}-\frac{\ell(1-x)(1-p)}{m} = \alpha. \label{eq:determine p}
\end{align}
For $p=0$, the first term is positive while the second is negative, and vice versa for $p=1$. Therefore the solution for $p$ lies indeed in $(0,1)$ and we will not calculate it explicitly. Instead we eliminate $p$ to only calculate the value of $\alpha$:
\begin{align*}
    &\frac{1}{qrm}\left( qr + (q-1)sm + qm + qr\ell(1-x) \right) \alpha \\
    &= \left( \frac{1}{m} + \frac{(q-1)s}{qr} \right) \left(\frac{1}{r}p+\frac{\ell(1-x)}{m}p - \frac{\ell(1-x)}{m}\right) + \left( \frac{1}{r} + \frac{\ell(1-x)}{m} \right) \left(\frac{1}{m}-\frac{1}{m}p-\frac{(q-1)s}{qr}p \right)\\
    &= -\frac{(q-1)\ell s(1-x)}{qrm} + \frac{1}{rm},
\end{align*}
reducing to
\begin{equation*} 
 \alpha = \frac{q-(q-1)\ell s(1-x)}{qr+(q-1)sm + qm + qr\ell(1-x)}.   
\end{equation*}
Since $x>1-\frac{q}{(q-1)\ell s}$, this value of $\alpha$ is positive. Plugging in $1-x=\frac{q}{(q-1)\ell s}-\xi$, we continue to find a bound on $\alpha^{-1}$:
\begin{align*}
    \alpha^{-1} &= \frac{qr+(q-1)sm + qm + qr\ell(1-x)}{q-(q-1)\ell s(1-x)} = \frac{qr+(q-1)sm + qm + qr\ell(1-x)}{(q-1)\ell s\xi}\\
    &< \frac{qr+(q-1)sm + qm + \frac{q^2r}{(q-1)s}}{(q-1)\ell s\xi} \leq \frac{2(m+r)}{\ell}\xi^{-1}. 
\end{align*}
This finishes the proof.
\end{proof}

\section{Computing the partition function using the Markov chain sampler}
\label{se:count}
In this section we prove Theorem~\ref{thm:main counting}.
We will do this with a self-reducibility argument, making use of a connection between removing and contracting edges.

We have the following result.
\begin{prop}\label{prop:sampling implies counting}
Let $x\in [1/3,1]$ and let $q\in \mathbb{N}_{\geq 2}$. 
Let $\mathcal{G}$ be a family of graphs which is closed under contracting edges.
Assume we are given an algorithm that for $n$-vertex and $m$-edge graph $G\in \mathcal{G}$ and any $\delta>0$ computes a random $\mathbb{Z}_q$-flow with distribution $\delta$-close to $\mu_{\mathrm{flow}}$ in time bounded by $T(\delta,n,m)$.
Then there is an algorithm that given an $n$-vertex and $m$-edge graph $G\in \mathcal{G}$ and any $\epsilon>0$ computes a number $\zeta$ such that with probability at least $3/4$ 
\[
e^{-\epsilon}\leq \frac{\zeta}{Z_{\mathrm{flow}}(G;q,x)}\leq e^{\epsilon}
\]
in time $O(n^2\epsilon^{-2}T( \epsilon/n,n,m))$.
\end{prop}

Before proving the proposition, let us show how it implies Theorem~\ref{thm:main counting}.
\begin{proof}[Proof of Theorem~\ref{thm:main counting}]
We prove part (i):  part (ii) follows in exactly the same way.
Fix positive integers $\iota$ and $d$ with $d$ at least $2$.
Consider the class of graphs $\mathcal{G}$ that have a basis for the cycle space $\mathcal{C}$ consisting of even sets satisfying $\iota(\mathcal{C})\leq \iota$ and $d(\mathcal{C})\leq d.$
By Lemma~\ref{lem:contraction is good} this class is closed under contracting edges.
By Theorem~\ref{thm:mixing time flow MC} we have an algorithm that for each $m$-edge graph $G\in \mathcal{G}$ and any $\delta>0$ computes a random $\mathbb{Z}_q$-flow with distribution within total variation distance $\delta$ from $\mu_{\mathrm{flow}}$ in time bounded by $T(\delta,n,m) =O(m^2 \log(m \delta^{-1}))$ provided $x>1-\frac{2}{(d+1)\iota}\geq 1/3$; see Remark~\ref{re:runtime}). 
The theorem now follows from the previous proposition combined with the fact that $Z_{\mathrm{flow}}(G;q,x)=(1-x)^{|E|}q^{-|V|}Z_{\mathrm{Potts}}(G;q,\tfrac{1+(q-1)x}{1-x})$ by Lemma~\ref{lem:tutte}. 
The running time is given by $O(n^2m^2 \epsilon^{-2} \log(nm \epsilon^{-1}))$.
\end{proof}

We now turn to the proof of Proposition~\ref{prop:sampling implies counting}.
\begin{proof}[Proof of Proposition~\ref{prop:sampling implies counting}]
As already mentioned above the proof relies on a self-reducibility argument.

The flow partition function satisfies the following well known deletion-contraction relation: for a graph $G=(V,E)$ and $e\in E$ not a loop, we have
\begin{equation}\label{eq:deletion contraction}
    Z_{\mathrm{flow}}(G;q,x)=(1-x)Z_{\mathrm{flow}}(G\setminus e;q,x) + xZ_{\mathrm{flow}}(G/e;q,x).  
\end{equation}
This holds because the collection of all flows on $G$ and on $G/e$ are in bijection with each other, while the flows on $G\setminus e$ correspond to the flows on $G$ that take value $0$ on $e$.

We rewrite~\eqref{eq:deletion contraction} as
\begin{equation}\label{eq:ratio flow}
\frac{Z_{\mathrm{flow}}(G/e;q,x)}{Z_{\mathrm{flow}}(G;q,x)}=\frac{1}{x}-\frac{1-x}{x}\cdot\frac{{Z_{\mathrm{flow}}(G\setminus e;q,x)}}{Z_{\mathrm{flow}}(G;q,x)},
\end{equation}
 and we interpret the fraction 
\[
\frac{{Z_{\mathrm{flow}}(G\setminus e;q,x)}}{Z_{\mathrm{flow}}(G;q,x)}
\]
as the probability that $e$ is assigned the value $0\in \mathbb{Z}_q$ when a flow is sampled from $\mu_{\mathrm{flow}}$. 
This probability can be estimated using the assumed sampler. Hence we can use the sampler to estimate \eqref{eq:ratio flow}.

From $G=(V,E)$. we now construct a series of graphs $G=G_0,G_1,\ldots G_t$ where in each step we contract one edge (which is not a loop). We can do this, until every component has been contracted to a single vertex, possibly with some loops attached to it. 
This takes $t=|V|-c(G)\leq |V|$ steps, where $c(G)$ denotes the number of components of $G$. 
In the end we have $|E|-|V|+c(G)\leq |E|$ edges (loops) left and 
the resulting graph $G_t$ thus has flow partition function  $Z_{\mathrm{flow}}(G_t;q,x)=(1+(q-1)w)^{|E|-|V|+c(G)}$.
Therefore
\begin{equation}\label{eq:telescope}
  \frac{(1+(q-1)x)^{|V|-c(G)}}{Z_{\mathrm{flow}}(G;q,x)} =\frac{Z_{\mathrm{flow}}(G_{r};q,x)}{Z_{\mathrm{flow}}(G_0;q,x)}= \frac{Z_{\mathrm{flow}}(G_1;q,x)}{Z_{\mathrm{flow}}(G_0;q,x)}\cdots \frac{Z_{\mathrm{flow}}(G_{t};q,x)}{Z_{\mathrm{flow}}(G_{t-1};q,x)}.   
\end{equation}

Note that for each $i$ and any non-loop edge $e\in E(G_i)$ we have by~\eqref{eq:ratio flow},
\begin{equation}\label{eq:bound prob e notin}
1\leq \frac{Z_{\mathrm{flow}}(G_i/ e;q,x)}{Z_{\mathrm{flow}}(G_i;q,x)}\leq 1/x\leq 3,
\end{equation}
since $x\geq1/3$. 

We can now estimate each individual probability on the right-hand side of \eqref{eq:telescope} to get an estimate for $Z_{\mathrm{flow}}(G;q,w)$. 
This is rather standard and can be done following the approach in~\cite{Jbook} for matchings. We therefore only give a sketch of the argument, leaving out technical details.

For each $i$, let
\[
p_i\coloneqq\frac{Z_{\mathrm{flow}}(G_i\setminus e;q,x)}{Z_{\mathrm{flow}}(G_i;q,x)}.
\]
To estimate $p_i$ we run our sampler $M=O(\epsilon^{-2}t)$ times with $\delta=O(\epsilon/t)$ to generate independent random flows $f_j$ $(j=1,\ldots,M)$. Denote by $X_j$ the random variable that is equal to $1$ if $e$ is not contained in $\supp(f_j)$ and $0$ otherwise.
We are in fact not interested in $p_i$, but rather in
\[
\hat{p}_i\coloneqq\frac{Z_{\mathrm{flow}}(G_i/ e;q,x)}{Z_{\mathrm{flow}}(G_i;q,x)}=\frac{1}{x}-\frac{1-x}{x}p_i.
\]
We therefore define the random variable $Y_j\coloneqq\frac{1}{x}-\frac{1-x}{x}X_j$ and $Y^i\coloneqq1/M\sum_{j=1}^M Y_j$.
Note that $\mathbb{E}[Y^i]=\mathbb{E}[Y_j] = \hat{p}_i$ and it is easy to check that $\text{Var}[Y^i]=1/M\text{Var}[Y_j]=1/M(\mathbb{E}(Y_j)-1)(1/x-\mathbb{E}(Y_j))$ for any $j=1,\ldots,M$.
We note that, by definition of the total variation distance, the fact that $x\geq 1/3$, and \eqref{eq:bound prob e notin}, we have 
\begin{equation}\label{eq:delta bound}
\hat{p}_i(1-2\delta)\leq \hat{p}_i-\frac{1-x}{x}\delta\leq \mathbb{E}[Y^i] = \mathbb{E}[Y_j] \leq \hat{p}_i+\frac{1-x}{x}\delta\leq (1+2\delta)\hat{p}_i.
\end{equation}
This implies that 
\[
\frac{\text{Var}[Y^i]}{\mathbb{E}[Y^i]^2}=\frac{1}{M}
\frac{(\mathbb{E}(Y_j)-1)(1/x-\mathbb{E}(Y_j))}{\mathbb{E}[Y^i]^2}\leq O(\epsilon^2/t).
\]

Consider next the random variable $Y\coloneqq\prod_{i=1}^t Y^i$. This will, up to a multiplicative factor (cf.\eqref{eq:telescope}), give us the desired estimate. 
Since the $Y^i$ are independent we have
\begin{align*}
\frac{\text{Var}[Y]}{\prod_{i=1}^t \mathbb{E}[Y^i]^2}=\prod_{i=1}^t \frac{\mathbb{E}[(Y^i)^2]}{\mathbb{E}[Y^i]^2}-1
=\prod_{i=1}^t \left(1+\frac{\text{Var}[Y^i]}{\mathbb{E}[Y^i]^2}\right)-1\leq O(\epsilon^2).
\end{align*}
Then by Chebychev's inequality $Y$ does not deviate much from $\prod_{i=1}^t \mathbb{E}[Y^i]$ with high probability, which by \eqref{eq:delta bound} and our choice of $\delta$ does not deviates much from $\prod_{i=1}^t \hat{p}_i$. 
More precisely, $Y$ will not deviate more than an $\exp(O(\epsilon))$ multiplicative factor from $\prod_{i=1}^t \hat{p}_i$ with high probability, as desired.

We need to access the sampler $O(t/\epsilon^2)$ many times with $\delta=O(\epsilon/t)$ to compute each $Y^i$. So this gives a total running time of $O(n^2 \epsilon^{-2} T(\epsilon/n, n,m))$.
This concludes the proof sketch.
\end{proof}


\section{Slow mixing of the flow chain}
\label{se:slow}
In this section we show that the flow Markov chain cannot mix rapidly for all $x \in (0,1)$. We do this by 
using the duality of our Markov chain on flows and Glauber dynamics of the Potts model on the planar grid (although the duality holds more generally on planar graphs). A result of Borgs, Chayes, and Tetali \cite{BCT12} for slow mixing of the Glauber dynamics of the Potts model on the grid (below a critical temperature) then immediately implies slow mixing of our flows Markov chain at the same temperature.

Given a graph $G=(V,E)$, let $\mathcal{F}_q(G)$ be the set of $\Z_q$-flows on $G$ and let $\Omega_q(G)$ be the set of $\tau: V \rightarrow [q]$ of $q$-spin configurations on $G$. Clearly $|\Omega_q(G)| = q^{|V|}$ and, as noted earlier, $|\mathcal{F}_q(G)| = q^{|E| - |V|-1}$.

Recall that the Glauber dynamics for the $q$-state Potts model for a graph $G$ and parameter $x$ is the following Markov chain with state space $\Omega_q(G)$. Given that we are currently at state $\sigma \in \Omega_q(G)$, we pick a vertex $v \in V$ uniformly at random and update its state as follows: we choose the new state to be $i$ with probability $x^{m(i)}/Z_v$, where $m(i)$ is the number of neighbours of $v$ that have state $i$ in $\sigma$, and $Z_v = \sum_i x^{m(i)}$.

Let $G=(V,E)$ be the $((L+1) \times (L+1))$-grid and $H=(V',E')$ the $(L \times L)$-grid. One can easily check that $|V'| = |E| - |V| + 1$ and so $|\Omega_q(H)| = |\mathcal{F}_q(G)|$. There is a natural bijection $\phi: \Omega_q(H) \rightarrow \mathcal{F}_q(G)$ defined as follows. First note that $H$ is the planar dual of $G$ (ignoring the outer face of $G$). Using this, write $v_1, \ldots, v_{L^2}$ for the vertices of $H$ and $C_1, \ldots, C_{L^2}$ for the corresponding faces (i.e.\ $4$-cycles) of $G$. Given $\sigma \in \Omega_q(H)$, let $\phi(\sigma) = \sum_{i = 1}^{L^2}\sigma(v_i)C_i$. We see that $\phi$ is injective since the $C_i$ form a basis of the cycle space of $G$, and hence $\phi$ must be bijective.

Now it is easy to check that the $q$-state Potts Glauber dynamics on $H$ is equivalent to the $\mathbb{Z}_q$-flow Markov chain on $G$  (where both chains have the same interaction parameter, say $x$) via the correspondence $\phi$ between their state spaces. In other words if $P$ and $Q$ are their respective transition matrices then $P_{\sigma_1 \sigma_2} = Q_{\phi(\sigma_1) \phi(\sigma_2)}$ for all $\sigma_1, \sigma_2 \in \Omega_q(H)$.

Borgs, Chayes, and Tetali \cite{BCT12} showed that the mixing time of the Glauber dynamics of the $q$-state Potts model on the $L \times L$ grid with interaction parameter $x = e^{-\beta}$ is bounded below by $x^{CL}$ for some constant $C$ when $\beta$ is above the critical threshold for the grid, i.e.\ $\beta \geq \beta_0(\Z^2) = \frac{1}{2} \log q + O(q^{-1/2})$. In particular this shows the same exponential lower bound on the mixing time  for the $\mathbb{Z}_q$-flow Markov chain (for the same interaction parameter $x$) on the $(L+1) \times (L+1)$-grid.  

\subsection*{Acknowledgement}
We are grateful to an anonymous referee for several useful suggestions.

\bibliographystyle{plain}
\bibliography{markov}

\end{document}